\documentclass[10pt]{amsart}
\usepackage{amsmath,amssymb,enumerate}

\newtheorem{thm}{Theorem}[section]
 \numberwithin{equation}{section} 
 \numberwithin{figure}{section} 
 \theoremstyle{plain}
 \theoremstyle{plain}    
 \newtheorem{cor}[thm]{Corollary} 
 \theoremstyle{plain}    
 \newtheorem{prop}[thm]{Proposition} 
 
\newtheorem{defi}[thm]{Definition}
 \newtheorem{lem}[thm]{Lemma} 
 \newtheorem{rem}[thm]{Remark}
\newtheorem{ex}[thm]{Example}

\newcommand{\N}{\mathbb{N}}
\newcommand{\R}{\mathbb{R}}
\newcommand{\C}{\mathbb{C}}
\newcommand{\Cc}{\mathcal{C}}
\newcommand{\f}{\varphi}
\newcommand{\p}{\psi}
\newcommand{\vep}{\varepsilon}
\newcommand{\Ec}{\mathcal{E}}
\newcommand{\EcX}{\mathcal{E}(X,\omega)}
\newcommand{\Hc}{\mathcal{H}}
\newcommand{\Sc}{\mathcal{S}}
\newcommand{\ind}{{\bf 1}}

\newcommand{\setdef}{\ \big\vert \ }
\newcommand{\Capa}{{\rm Cap}}
\newcommand{\Capo}{{\rm Cap}_{\omega}}

\newcommand{\Capis}{{\rm Cap}_{\psi}}
\newcommand{\Capfp}{{\rm Cap}_{\f,\psi}}
\newcommand{\Capip}{{\rm Cap}_{\varphi,\psi}}
\newcommand{\Capipj}{{\rm Cap}_{\varphi,\psi_j}}
\newcommand{\MA}{\mathrm{MA}\,}
\newcommand{\vol}{{\rm Vol}}

\newcommand{\psh}{{\rm PSH}}
\newcommand{\pshXo}{{\rm PSH}(X,\omega)}

\begin{document}
\title{Generalized Monge-Amp\`ere Capacities} 
\date{\today\\
The authors are partially supported by the french ANR project MACK} 
\author{E. Di Nezza} 

\address{Institut Math\'ematiques de Toulouse \\
Universit{\'e} Paul Sabatier\\ 31062 Toulouse\\ France}

\email{eleonora.dinezza@math.univ-toulouse.fr}

\author{Chinh H. Lu}
\address{Chalmers University of Technology \\ Mathematical Sciences\\
412 96 Gothenburg\\ Sweden}

\email{chinh@chalmers.se}

\begin{abstract}
We study various capacities on compact K\"{a}hler manifolds which generalize
the Bedford-Taylor Monge-Amp\`ere capacity. 
We then use these capacities to study the existence and the regularity of solutions of complex Monge-Amp\`ere equations.
\end{abstract}

\maketitle
\tableofcontents

\section{Introduction}
Let $(X,\omega)$ be a compact K\"{a}hler manifold of complex 
dimension $n$ and let $D$ be an arbitrary divisor on $X$. 
Consider the complex Monge-Amp\`ere equation 
\begin{equation}\label{eq: intro 1}
 (\omega+dd^c \varphi)^n=f\omega^n, 
\end{equation}
where $0\leq f\in L^1(X)$ is such that $\int_X f\omega^n=\int_X \omega^n$. 
It follows from \cite{GZ07} and 
\cite{Di09} that equation (\ref{eq: intro 1})  has a unique normalized solution in the finite energy class $\Ec(X,\omega)$.
We say that the solution $\f$ is normalized if $\sup_X \f=0$. 

If $f$ is strictly positive and smooth on $X$, we know from the seminal paper of Yau \cite{Y78} that the solution is also smooth on  $X$. 
Recall that this solves in particular the Calabi conjecture and allows to construct Ricci flat metrics on $X$ whenever $c_1(X)=0$.

Given $f$ positive and  smooth on $X\setminus D$, it is natural to investigate the regularity of the solution. In \cite{DL1} 
we have proved in many cases that the solution $\f$ is smooth 
in $X\setminus D$.

As in the classical case of Yau \cite{Y78}, the most difficult step is to establish an a priori $\Cc^0$-estimate. 
This estimate is much more difficult in our situation since in general the solution is not globally bounded. 
A natural idea is to bound the normalized solution from below by a singular quasi plurisubharmonic function (qpsh for short). 
This is where generalized Monge-Amp\`ere capacities play a crucial role.

We recall the notion of the classical capacity 
$\Capa_{\omega}$ introduced and studied in \cite{Kol03} and \cite{GZ05}: 
$$
\Capa_{\omega}(E)= \sup\left\{\int_E (\omega+dd^c u)^n\ \setdef \ u\in \psh(X,\omega), \ -1\leq u\leq 0\right\}, \ E\subset X.
$$ 

A strong comparison between the Lebesgue measure and $\Capa_{\omega}$, as is needed in a celebrated method due to Ko{\l}odziej \cite{Kol98}, 
does not hold in our setting. We therefore study other capacities to provide an a priori $\Cc^0$-estimate. In dealing with complex 
Monge-Amp\`ere equations in quasiprojective varieties we were naturally lead to work with generalized capacities of type 
$\Capa_{\psi-1, \psi}$ in \cite{DL1} (see below for their definition).

\medskip

In this paper, we make a systematic study of these capacities as well as the more general $\Capa_{\f, \psi}$ capacities: 
let $\f,\p$ be two $\omega$-plurisubharmonic functions on $X$ such that $\f<\p$ on $X$ modulo possibly a pluripolar set.  
The $(\f,\p)$-Capacity of a Borel subset $E\subset X$ is defined by
\begin{equation*}
\Capa_{\f,\p} (E):=\sup\left\{\int_E (\omega+dd^c u)^n\  \setdef \ 
u\in \psh(X,\omega),\ \f\leq u\leq \p \right \}.
\end{equation*}
Here, for a $\omega$-psh function $u$,  
$(\omega+dd^c u)^n$ is the non-pluripolar Monge-Amp\`ere measure 
of $u$. See Section 2 for the definition.
 When $\f\equiv \p-1$, we drop the index $\f$ and denote the 
$(\p-1,\p)$-Capacity by $\Capis$,
$$
\Capis: = \Capa_{\p-1,\p}.
$$
This is exactly the generalized 
capacity used in our previous paper \cite{DL1}.  If moreover $\p$ is constant, $\p\equiv C$, 
we recover the Monge-Amp\`ere capacity defined above
$$
\Capa_{C} = \Capo.
$$

Given any subset $E\subset X,$ we define  the outer $(\f,\p)$-capacity of  $E$ by
\begin{equation*}
\Capa_{\f,\p}^*(E):=\inf \left\{ \Capa_{\f,\p}(U) \setdef U\ \text{is an open subset of } X,\  E\subset U\right\}.
\end{equation*}
We say that the $(\f,\p)$-capacity \textit{characterizes pluripolar sets} on $X$ if for any subset $E\subset X$, the
following holds
$$
\Capfp^*(E)=0\Longleftrightarrow \ \text{E is a pluripolar subset of}\ X.
$$ 
If $E\subset X$ is a Borel subset we set
\begin{equation*}
h_{\f,\p,E}(x):=\sup\left\{ u(x) \setdef u\in \psh(X,\omega), u\leq \p\ {\rm on}\ X,\ u\leq \f \; \text{q.e.}\ E\right\}.
\end{equation*}
Here, quasi everywhere (q.e. for short) means outside a pluripolar set. 
Let $h_{\f,\p,E}^*$ be its upper semicontinuous regularization which we call the $(\f,\p)$-extremal function of $E$. 
We establish a useful characterization of the $(\f,\p)$-capacity in terms of the relative extremal function for any subset.
\medskip

When $\f$ belong to the finite energy class $\EcX$ we can bound $\Capfp$ by $F(\Capo)$ for some positive function $F$
 which vanishes at $0$. This uniform control turns out to be very useful in studying convergence of the complex Monge-Amp\`ere operator 
 since it allows us to replace quasi-continuous functions by continuous ones without affecting the final result. 
We also prove that  the generalized Monge-Amp\`ere capacity $\Capfp$ characterizes pluripolar sets when the lower  weight is in  $\EcX$:

\medbreak

\noindent {\bf Theorem A.} {\it Assume that $\f\in \EcX$ and 
$\p\in \psh(X,\omega)$ such that $\f<\p$ modulo a pluripolar subset. }
\begin{itemize}
\item[(i)]{\it Let $E\subset X$ be a Borel subset of $X$, and denote by $h_E$ the $(\f,\p)$-extremal function of $E$. 
The  outer $(\f,\p)$-capacity of $E$ is given by}
\begin{equation*}
\Capfp^*(E)=\int_{\{h_E<\f\}}\MA(h_E)=\int_X \left( \frac{-h_E+\p}{-\f+\p}\right) \MA(h_E),
\end{equation*}
where $h_E:=h^*_{\f,\p,E}$ is the $(\f,\p)$-extremal function of $E$.
\item[(ii)] {\it There exists a function $F:\R^+\to \R^+$ such that $\lim_{t\to 0^+}F(t)=0$ and such that for all Borel subset $E$,}
$$
\Capfp(E) \leq F(\Capo(E)). 
$$
\item[(iii)]{\it $\Capfp$ characterizes pluripolar sets.}
\end{itemize}
\medbreak

\noindent We  stress that the function $F$ in $(ii)$ is quite explicit (see Theorem \ref{thm: comparison of Capacity}).

\medskip

As we have  underlined, these generalized capacities play an important role in studying
complex Monge-Amp\`ere equations on quasi-projective varieties (see \cite{DL1}). 
We give in the second part of this paper several other applications. 
\medbreak

We consider the following complex Monge-Amp\`ere equation
\begin{equation}\label{eq: MAeq lambda}
(\omega+dd^c \f)^n =e^{\lambda \f} f\omega^n, \ \lambda \in \R.
\end{equation}
Assume that $0<f\in \Cc^{\infty}(X\setminus D)$ satisfies Condition 
$\Hc_f$, i.e. $f$ can be written as
$$
f=e^{\p^+-\p^-}, \ \ \p^{\pm} \ {\rm are \ quasi \ psh
 \ functions\ on }\ X\ , \ \p^-\in L^{\infty}_{\rm loc}(X\setminus D).
$$

When $\lambda=0$ and $f$ satisfies $\int_X f\omega^n=\int_X \omega^n$, 
we proved in  \cite{DL1} that there is a unique normalized solution 
in $\EcX$ which is smooth on $X\setminus D$. When $\lambda>0$ and 
$\int_X f\omega^n<+\infty$ the same result holds since the $\Cc^0$ estimate follows easily from the comparison principle.

Consider now  the case when $\lambda<0$. In this case solutions do not always exist and 
when they do, there may be many of them. Our result here says that
any solution in $\EcX$ ({\it if exists}) is smooth on $X\setminus D$.

\medskip

\noindent{\bf Theorem B.} {\it Let $0<f\in \Cc^{\infty}(X\setminus D) \cap  L^1(X)$. Assume that $f$ satisfies Condition $\Hc_f$ and $\f\in \EcX$ is a solution of 
$$  
(\omega+dd^c \f)^n =e^{\lambda \f} f\omega^n, \ \lambda <0 .
$$
Then  $\f$ is smooth on $X\setminus D$.} 

 \medskip
 
Note that when $\lambda<0$ and equation (\ref{eq: MAeq lambda}) has a solution in $\EcX$, the measure 
$\mu=f\omega^n$ is dominated by $\MA(u)$ for some $u\in \psh(X,\omega)
\cap L^{\infty}(X)$. In particular, $f\in L^1(X)$.

\medskip

We next investigate the case when $\lambda>0$ and $f$ is not integrable on $X$. Of course solutions do not always exist. But observe that when $\f$ is
 singular enough $e^{\f}f$ will be integrable on $X$ and it is then reasonable to find  a solution. For example, one can look at densities of the type
$$
f\simeq \frac{1}{\vert s\vert^2},
$$
which is not integrable. Here $s$ is a holomorphic section of the line bundle associated to $D$. Such densities have been considered by Berman and Guenancia in their study of the compactification of the moduli space of canonically polarized manifolds \cite{BG13}. They have shown that there exists a unique solution $\f\in \EcX$ which is smooth in $X\setminus D$.
As another application of the generalized Monge-Amp\`ere capacities 
we show in the following result that in a general context whenever a solution in $\EcX$ exists it is smooth outside $D$.

\medskip

\noindent{\bf Theorem C.} {\it Assume $0<f\in \Cc^{\infty}(X\setminus D)$ satisfies Condition $\Hc_f$. If the equation 
$$ (\omega+dd^c \f)^n =e^{\lambda \f} f\omega^n, \ \lambda >0$$
admits a solution $\f\in \EcX$ then $\f$ is smooth on 
$X\setminus D$.}  
 
\smallskip
\noindent Let us stress  that in Theorem C we do not assume that 
$\int_X f\omega^n<+\infty$. It turns out that the existence of solutions in 
$\EcX$ is equivalent to the existence of subsolutions in this class, these are easy to construct in concrete situations (see Example \ref{ex: subsol}). We also obtain a similar result in the case of semipositive and big classes (see Theorem \ref{thm: thmC semi} and Example \ref{ex: subsol semi}). 

\medskip

Finally we use generalized capacitites to study the critical integrability of a given $\phi\in \psh(X,\omega)$.

\smallskip

\noindent{\bf Theorem D.} {\it Let $\phi\in \psh(X,\omega)$ and $\alpha=\alpha(\phi)\in (0,+\infty)$ be the canonical threshold of $\phi$, i.e. 
$$
\alpha=\alpha(\phi):= \sup\{t>0 \ \setdef \ e^{-t\phi} \in L^1(X)\}.
$$
Then there exists $u\in \psh(X,\omega)$ with zero Lelong number at all points such that $e^{u-\alpha \phi}$ is integrable. Moreover, there exists a unique $\f\in \EcX$ such that 
$$
(\omega +dd^c \f)^n = e^{\f-\alpha \phi} \omega^n.
$$
}
It turns out that one can even chose $u=\chi\circ \phi$ in $\Ec(X,\omega)$, as an explicit function of $\phi$ with attenuated singularities (see Theorem \ref{thm: critical integrability}).
\medskip

The paper is organized as follows. In section 2 we recall some known facts on energy classes, we introduce 
generalized capacities on compact K\"{a}hler manifolds and prove Theorem A. As an application of the generalized
 capacities we give another proof of the domination principle in 
 $\EcX$ in Section 3. In Section 4 we use generalized 
 capacities to study complex 
Monge-Amp\`ere equations as (\ref{eq: MAeq lambda}). The proof of Theorem D will be given in Section 4 as well.
\bigskip

\noindent {\bf Acknowledgements.} We would like to thank 
Vincent Guedj and Ahmed Zeriahi for constant help, many suggestions and encouragements. 
We  also thank Robert Berman and  Bo Berndtsson for useful discussions.
We are indebted to Henri Guenancia for a careful reading and very useful 
comments on a previous draft version of this paper.
 
\section{Generalized Monge-Amp\`ere Capacities}\label{sect: basic properties}
Let $(X,\omega)$ be a compact K\"{a}hler manifold of complex dimension $n$. In this section we prove some basic properties of the 
$(\f,\psi)$-capacity and of the relative 
$(\f,\psi)$-extremal functions. 
 
\subsection{Energy classes}
\begin{defi}
We let $\psh(X,\omega)$ denote the class of $\omega$-plurisubharmonic functions ($\omega$-psh for short) on $X$, i.e. the class of functions 
$\f$ such that locally  $\f= \rho+ u$, where $\rho$ is a local potential of $\omega$ and $u$ is a plurisubharmonic function.
\end{defi}

Let $\f$ be some unbounded $\omega$-psh function on $X$ 
and consider $\f_j:=\max(\f, -j)$ the "canonical approximants". It has been shown in \cite{GZ07} that 
$$
\ind_{\{\f_j >-j \}} (\omega+dd^c \f_j)^n
$$ 
is a non-decreasing sequence of Borel measures. We denote its limit by
$$
\MA(\f) = (\omega+dd^c \f)^n := \lim_{j\to +\infty} \ind_{\{\f_j >-j \}} (\omega+dd^c \f_j)^n.
$$

\begin{defi} 
We denote by $\Ec(X,\omega)$ the set of $\omega$-psh  functions having full Monge-Amp{\`e}re mass:
$$
\EcX:= \left\{ \f\in \psh(X,\omega)\ \setdef \ \int_X \MA(\f) =\int_X \omega^n\right\}.
$$
\end{defi}

Let us stress that $\omega$-psh functions with full Monge-Amp{\`e}re mass have mild singularities.
In particular, any $\f\in\Ec(X,\omega)$ has zero Lelong numbers $\nu(\f, \cdot)=0$ (see \cite[Corollary 1.8]{GZ07}). We also recall that, for 
every $\f\in\Ec(X,\omega)$ and any $\p\in \psh(X,\omega)$, the \emph{generalized comparison principle} is valid, 
namely 
$$ 
\int_{\{\f<\psi\}} (\omega+dd^c  \psi)^n \leq \int_{\{\f<\psi\}} (\omega+dd^c \f)^n.
$$
\begin{defi} 
Let $\chi:\R^-\to\R^-$ be an increa\-sing function such 
that $\chi(0)=0$ and $\chi(-\infty)=-\infty$. 
We denote by $\Ec_\chi(X,\omega)$ the class of $\omega$-psh functions having 
finite $\chi$-energy:
$$
\Ec_{\chi}(X,\omega):= \left\{\f\in \Ec(X,\omega) \;\,  | \;\, \chi(-|\f|)\in L^1(\MA(\f)) \right\}.
$$ 
\end{defi} 

For $p>0$, we  use the notation
$$
{\Ec}^p(X,\omega):=\Ec_{\chi}(X,\omega),
\text{ when } \chi(t)=-(-t)^p.
$$

\subsection{The $(\f,\p)$-Capacity}
In this subsection we always assume that 
 $\f,\p\in \psh(X,\omega)$ are such that   $\f<\p$ quasi everywhere 
 on $X$. The $(\f,\psi)$-capacity of a Borel subset $E\subset X$ is defined by
$$
\Capfp(E):= \sup\left\{ \int_E \MA (u) \ \setdef \ u\in \psh(X, \omega) , \ \f\leq u\leq \p\right\}.
$$
When $\f\equiv \p-1$, to simplify the notation we simply denote 
$$
\Capis:= \Capa_{\p-1,\p}.
$$
If moreover $\p\equiv C$ is constant we recover the Monge-Amp\`ere capacity introduced in \cite{BT82}, \cite{Kol03}, \cite{GZ05}.
The following properties of the $(\f,\psi)$-Capacity follow straightforward from the definition.
\begin{prop}
(i) If $E_1\subset E_2\subset X$ then $\Capfp(E_1)\leq \Capfp(E_2)$ .

(ii) If $E_1, E_2, \cdots$ are Borel subsets of $X$ then 
$$
\Capfp\left(\bigcup_{j=1}^{\infty} E_j\right)\leq \sum_{j=1}^{+\infty} \Capfp (E_j) .
$$

(iii) If $E_1\subset E_2\subset \cdots$ are Borel subsets of $X$ then
$$
\Capfp \left(\bigcup_{j=1}^{\infty} E_j\right) = \lim_{j\to+\infty} \Capfp (E_j) .
$$
\end{prop}
The \textit{outer $(\f,\p)$-capacity} of  $E$ is defined by
\begin{equation*}\label{eq: outer capi def}
\Capa_{\f,\psi}^*(E):=\inf \left\{ \Capfp(U) \setdef U\ \text{is an open subset of } X,\  E\subset U\right\}.
\end{equation*}
We say that the $(\f,\psi)$-capacity \textit{characterizes pluripolar sets} on $X$ if for any subset $E\subset X$, the
following holds
$$
\Capfp^*(E)=0\Longleftrightarrow \ \text{E is a pluripolar subset of}\ X.
$$ 

\begin{defi}
If $E\subset X$ is a Borel subset we set
\begin{equation*}
h_{\f,\p,E}:=\sup\left\{ u\in \psh(X,\omega),\ u\leq \f \ \text{quasi\ everywhere\ on}\ E, u\leq \p \ \text{on}\ X\right\},
\end{equation*}
where "quasi everywhere" means outside a pluripolar set. 
The upper semicontinuous regularization of $h_{\f,\p,E}$ is called the relative $(\f,\p)$-extremal function of $E$.
\end{defi}

\begin{prop}\label{prop: properties} Let $E\subset X$.

\begin {itemize}
\item[(i)] The function $h_{\f,\p,E}^*$ is $\omega$-psh. It satisfies $\f \leq h_{\f,\p, E}^* \leq \p$ on $X$
and $h_{\f,\p, E}^*=\f$ quasi everywhere on $E.$
\item[(ii)] If $P \subset E$ is pluripolar, then 
$h_{\f,\p, E \setminus P}^* \equiv h_{\f,\p,E}^*$; in particular $h_{\f,\psi,P}^* \equiv \p$.
\item[(iii)] If $(E_j)$ are  subsets of $X$ increasing towards $E \subset X$, then $h_{\f,\p, E_j}^*$ decreases towards 
$h_{\f,\p, E}^*$.
\item[(iv)] If $h^*_{\f,\psi,E}\equiv \p$ then $E$ is pluripolar.
\end{itemize}
\end{prop}
\begin{proof}
The statement $(i)$ is a standard consequence of Bedford-Taylor's work \cite{BT82}. Set $E_1:=E\setminus P$, and denote by 
$h=h^*_{\f,\p, E},\  h_1=h^*_{\f,\p, E_1}$ the corresponding $(\f,\p)$-extremal
functions of $E, E_1$. Since $E_1\subset E$ it is clear that $h_1\geq h$. 
On the other hand $h_1=\f$ quasi everywhere on $E_1$ hence on $E$. This yields $h_1\leq h$ whence equality.

Let us prove $(iii)$. Since $(E_j)$ is increasing, $h_j:=h^*_{\f,\psi,E_j}$ is decreasing toward $h\in \psh(X,\omega).$ 
It is clear that $h\geq h^*_{\f,\psi,E}$. By definition, for each $j\in \N$, $h_j=\f$  quasi everywhere on  $E_j$.
It then follows that  $h=\f$ quasi everywhere on $E$. We then infer that $h\leq h^*_{\f,\psi,E}$, hence the equality.

To prove $(iv)$ assume that $h^*_{\f,\p,E}\equiv\psi$. By definition of $h:=h^*_{\f,\p,E}$ and by Choquet's
 lemma we can find an increasing sequence $(u_j)$ such that $u_j=\f$ on $E$ and $\left(\lim_{j\rightarrow +\infty} u_j\right)^*=h$. Note that 
$$
E\subset \left\{\left(\limsup_{j\rightarrow +\infty} u_j\right)<\left(\limsup_{j\rightarrow +\infty} u_j\right)^*\right\} ,
$$ 
modulo a pluripolar set. The latter is also pluripolar, hence $E$ is  pluripolar.
\end{proof}

\smallskip

\begin{thm}\label{thm: pluripolar}
If $\f\in \EcX$ and $E\subset X$ is pluripolar then $\Capfp^*(E)=0$.
\end{thm}
\begin{proof}
Assume  that $\f\in \EcX$ and fix a pluripolar set $E\subset X$. By translating
$\p$ and $\f$ by a constant we can assume that $\p\leq 0$.
It follows from \cite[Proposition 2.2]{GZ07} that $\f\in \Ec_{\chi}(X,\omega)$ for some convex increasing function 
$\chi: \R^-\rightarrow \R^-$. We can find 
$u\in \Ec_{\chi}(X,\omega), u\leq 0$ such that 
$E\subset \{u=-\infty\}$. We claim that 
\begin{equation*}
\Capfp (\{u<-t\})\leq \frac{-2}{\chi(-t)}\left( E_{\chi}(u)+2^n E_{\chi}(\f)\right),\ \forall t>0.
\end{equation*}
Indeed, let $v\in PSH(X,\omega)$ such that
$\f\leq v\leq \p.$ We obtain immediately that
\begin{equation*}
\int_{\{u<-t\}} \MA(v)\leq \frac{1}{-\chi (-t)}\int_{\{u<-t\}}(-\chi\circ u ) \MA(v).
\end{equation*}
From this and \cite[Proposition 2.5]{GZ07} we get
\begin{equation*}
\int_{\{u<-t\}} MA(v)\leq \frac{-2}{\chi (-t)}\left( E_{\chi}(u)+ E_{\chi}(v)\right).
\end{equation*}
This coupled with  the \textit{fundamental inequality} in 
\cite[Lemma 2.3]{GZ07} yield the claim.
Since for any $t>0$, $E\subset \{u<-t\}$ we obtain 
$$
\Capfp^* (E)\leq \Capfp (u<-t)\to 0 \ \ \text{as}\ \ t\to+\infty.
$$ 
\end{proof}

From now on we fix $\f, \p$ two functions in $\EcX$ such that $\f<\p$
quasi everywhere on $X$.

Given any $u\in \pshXo$ such that $u\leq 0$, it follows from \cite[Example 2.14]{GZ07} (see also the Main Theorem in \cite{CGZ07}) that 
$u_p:=-(-u)^p$ belongs to $\EcX$ for any $0<p<1$. The same arguments can be applied to get the following result: 
\begin{lem}\label{lem: the two weights}
Let $\chi: \R^- \rightarrow \R^-$ be any measurable function. Assume that there exists  $q>0$ such that 
\begin{equation*}
\sup_{t\leq -1} |\chi(t)| (-t)^{-q} =C<+\infty.
\end{equation*}
Then for any $u\in \psh(X,\omega)$ such that $u\leq -1$ and any $0<p<\frac{1}{q+1}$ we have 
$$
\int_X |\chi\circ u_p| \MA(u_p) \leq A, 
$$
where $u_p:= -(-u)^p$ and $A$ is a positive constant depending only on $C,p,q$.
\end{lem}

\begin{proof}
In the proof we use $A$ to denote various positive constants which are under control. By considering $u^j:=\max(u,-j)$, the canonical approximants of $u$, and letting $j\to +\infty$ it suffices to treat the case when $u$ is bounded. 
We compute 
$$
\omega +dd^c u_p = \omega + p(1-p)(-u)^{p-2} du \wedge d^c u + p(-u)^{p-1}dd^c u.
$$
We thus get 
$$
0\leq \omega +dd^c u_p \leq (-u)^{p-1} (\omega + dd^c u) + \omega + (-u)^{p-2}
du \wedge d^c u.
$$
We need to verify the following bounds:
$$
\int_X |\chi\circ u_p | (-u)^{p-1}(\omega +dd^c u)^k\wedge \omega^{n-k} \leq A
$$
and 
$$
\int_X |\chi\circ u_p | (-u)^{p-2}du \wedge d^c u \wedge (\omega +dd^c u)^k\wedge \omega^{n-k-1} \leq A,
$$
where $k=0,1,...,n$.
Let us consider the first one. By assumption we have 
$$
|\chi \circ u_p|(-u_p)^{-q} \leq C.
$$
To bound the first term, it thus suffices to get a bound for
$$
\int_X (-u)^{p-1+pq} (\omega +dd^c u)^k\wedge \omega^{n-k},
$$
which is easy since $p+pq-1<0$. For the second one it suffices get a bound for 
$$
\int_X  (-u)^{p-2+pq} du \wedge d^c u \wedge (\omega +dd^c u)^k\wedge 
\omega^{n-k-1},
$$
which follows easily by the same reason and by integration by parts.

\end{proof}

We know from Theorem \ref{thm: pluripolar} that $\Capfp$ vanishes on 
pluripolar subsets of $X$. This suggests that $\Capfp$ is dominated by $F(\Capo)$, where $F$ is some positive function vanishing at $0$. The following result 
gives an explicit formula of $F$. 

\begin{thm}\label{thm: comparison of Capacity}
Let $\chi: \R^-\rightarrow \R^-$ be a convex increasing function and 
$\f \in \Ec_{\chi}(X,\omega)$. Let $q>0$ be a positive real number such that 
\begin{equation}\label{eq: the two weights}
\sup_{t\leq -1} |\chi(t)| (-t)^{-q} <+\infty.
\end{equation}
Then for any $p<\frac{1}{1+q}$ there exists $C>0$ depending on $p,q,\chi,\f$ such that
$$
\Capa_{\f,0}(K) \leq \frac{C}{\left |\chi \left( -\Capa_{\omega}(K)^{\frac{-p}{n}}\right)\right|} \ ,\ \forall K\subset X.
$$
\end{thm}
As a concrete example, when $\f\in \Ec^q(X,\omega)$ for some $q>0$ and  $p<1/(1+q)$, then we 
can take $F(s):=s^{\frac{pq}{n}}$ for $s>0$, getting 
$$
\Capa_{\f,0} (K) \leq C \,\Capa_{\omega}(K)^{\frac{pq}{n}}.
$$
\begin{proof}
Fix $p>0$ such that $p(q+1)<1$. Let $V_K$ be the extremal 
$\omega$-plurisubharmonic function of $K$:
$$
V_K:=\sup\{u \ \setdef \,u\in \psh(X, \omega),\, u\leq 0 \, \ \rm{on}\,\, K\},
$$
and $M_K:=\sup_X V_K^*$.  It follows from  (\ref{eq: the two weights}) and  Lemma \ref{lem: the two weights} that the function 
$$
u= -(-V_K^*+M_K+1)^p 
$$
belongs to $\Ec_{\chi}(X,\omega)$. 
Fix $h\in \psh(X,\omega)$ be such that $\f\leq h\leq 0$. 
It follows from Lemma \ref{lem: technical} below that 
$$
\int_X \vert \chi \circ u \vert \MA(h)\leq C_1,
$$
where $C_1>0$ only depends on $\chi$, $p, q$ and $\f$.
Therefore, using the fact that $V_K^*\equiv 0$ quasi everywhere on $K$ we  get 
\begin{eqnarray*}
\int_K \MA(h)\leq \int_X \frac{\vert \chi\circ u\vert}{\vert \chi(- M_K^p)\vert} \omega_h^n \leq \frac{C_1}{\vert \chi(-M_K^p)\vert}.
\end{eqnarray*}
It follows from \cite{GZ05} that $M_K\geq C_2 \Capa(K)^{-1/n}.$ This coupled with the above  yield the result.
\end{proof}

\begin{lem}\label{lem: technical}
Assume that $\chi$, $p,q$ and $\f$ are as in Theorem \ref{thm: comparison of Capacity}. 
Then there exists $C>0$ depending on $\chi,p,q, \f$ such that 
$$
\int_X \vert \chi(-(-u)^p)\vert \MA(v) \leq C, \ \forall u, v\in \psh(X,\omega), \ \sup_X u=-1, \ \f \leq v\leq 0.
$$
\end{lem}

\begin{proof}
We argue by contradiction, assuming that there are two sequences $(u_j)$, $(v_j)$
of functions in $\psh(X,\omega)$ such that $\sup_X u_j= -1$, $\f\leq v_j\leq 0$, and 
$$
\int_X \vert \chi(-(- u_j)^p)\vert \MA(v_j) \geq 2^{(n+2)j}, \ \forall j\in \N . 
$$
Set 
$$
u:=\sum_{j=1}^{+\infty} 2^{-j} u_j,\ v=\sum_{j=1}^{+\infty} 2^{-j} v_j.
$$
Then $u\in\psh(X,\omega)$, $u\leq -1$. Moreover, it follows from  Lemma \ref{lem: the two weights} that
$$
u_p:=-(-u)^p \in \Ec_{\chi}(X,\omega).
$$ 
We also have $\f\leq v\leq 0$, in particular 
$v\in \Ec_{\chi}(X,\omega)$. But
$$
\int_X \vert \chi\circ u_p \vert \MA(v) \geq \sum_{j=1}^{+\infty}2^j=+\infty,
$$
which contradicts \cite[Proposition 2.5]{GZ07}.
\end{proof}


\begin{prop}\label{prop: maximal}
Let  $E$ be a Borel subset of $X$ and set $h_E:=h^*_{\f,\psi,E}$ the relative $(\f,\p)$-extremal function of $E$. Then 
$$
\MA(h_E)\equiv 0 \ {\rm on}\ \{h_E<\p\}\setminus \bar{E}.
$$
\end{prop}
\begin{proof}
We first assume that $\p$ is continuous on $X$. Set $h:=h_E$ and let $x_0\in X\setminus \bar{E}$ 
be such that $(h-\p)(x_0)<0.$ Let $B:=B(x_0,r)\subset X\setminus \bar{E}$ be a small ball around $x_0$ 
such that $\sup_{\bar{B}}(h-\p)(x)=-2\delta<0.$ 
Let $\rho$ be a local potential of $\omega$ in $B.$ Shrinking $B$ a little bit we can assume that
 $\sup_{\bar{B}}\vert \rho\vert <\delta$ and ${\rm osc}_{\bar{B}} \psi<\delta/2$. 
By definition of $h$ and by Choquet's lemma we can find an increasing sequence $(u_j)_j\subset \Ec(X,\omega)$ such that 
$u_j= \f$ quasi everywhere on $E$, $u_j\leq \p$ on $X$, and $(\lim_j u_j)^*=h$.
For each $j, k\in \N$, we solve the Dirichlet problem to find 
$v_j^k\in \psh(X,\omega)\cap L^{\infty}(X)$ such that $\MA(v_j^k)=0$ in $B$ and $v_j^k\equiv \max(u_j,-k)$ on $X\setminus B$. 
Since 
$$
\rho+v_j^k\leq\rho+h\leq -\delta+\p\leq \sup_{\bar{B}}\p-\delta
$$
on $\partial B$, we deduce from the maximum principle that $v_j^k\leq \inf_{\bar{B}} \p -\delta/2-\rho\leq \p$ on $B$.
 Furthermore, taking
 $k$ big enough such that $\p\geq -k$, we can conclude that $v_j^k\leq\p$ on $X$. For $j\in \N$ fixed, by the comparison principle 
 $(v_j^k)_k$ decreases to $v_j\in \Ec(X,\omega)$. Then $u_j\leq v_j\leq h$ since $v_j=u_j=\f$ on $E$ and $v_j\leq \p$ on $X$. 
 It follows from \cite{GZ07} that the sequence of Monge-Amp\`ere measures $MA(v_j^k)$ converges weakly to $MA(v_j)$. 
 Thus $MA(v_j)(B)=0.$ On the other hand, $v_j$ increases almost everywhere to $h$ and these functions belong to $\Ec(X,\omega).$ 
 The same arguments as in  \cite[Theorem 2.6]{GZ07} show that $MA(v_j)$ converges weakly to $MA(h)$. We infer that $MA(h)(B)=0$.

It remains to remove the continuity hypothesis on $\p$. Let $(\p_j)$ be a sequence of continuous functions in 
$\psh(X,\omega)$ decreasing to $\p$ on $X$. Let $h_j:=h_{\f,\p_j,E}^*$ be the 
relative $(\f,\p_j)$-extremal function of $K$. Then $h_j$ decreases to $h$, hence $\MA(h_j)$ converges weakly to $\MA(h)$. 
Denote by $V:=\{h<\p\}\setminus \bar{E}$. Now, fix $\varepsilon>0$ and $U$ an open subset of $X$ such that 
$$
\Capa_{\omega}\left[(U\setminus V) \cup (V\setminus U)\right]\leq \varepsilon .
$$
From the first step we know that $\MA(h_j)$ vanishes on $V$. Thus
\begin{eqnarray*}
\int_V \MA(h) &\leq & \int_U \MA(h) + F(\varepsilon) \\
&\leq & \liminf_{j\to+\infty} \int_U \MA(h_j) + F(\varepsilon) \\
&\leq & \liminf_{j\to+\infty} \int_V \MA(h_j) +2 F(\varepsilon) \\
&=& 2 F(\varepsilon),
\end{eqnarray*}
It suffices now to let $\varepsilon\to 0$ since  $\lim_{\vep \to 0}F(\varepsilon)=0$ thanks to Theorem \ref{thm: comparison of Capacity}.
\end{proof}

\begin{lem}\label{lem: the easy inequality for cap formula}
Let $E\subset X$ be a Borel subset and $h_E:=h^*_{\f,\psi,E}$ be its relative $(\f,\p)$-extremal function. Then we have
$$
\Capfp (E) \leq \int_{\{h_E<\p\}} \MA(h_E) .
$$
\end{lem}
\begin{proof}
Observe first  that the $(\f,\p)$-capacity can be equivalently defined by
$$
\Capfp(E):= \sup\left\{ \int_E \MA (u) \ \vert \ u\in \psh(X, \omega) , \ \f < u\leq \p\right\}.
$$ 
For simplicity, set $h:=h_E$. Now take any $u\in \psh(X,\omega)$ such that $\f<u\leq \p$. Then
$$
E\subset \{h<u\}\subset \{h<\p\},
$$
where the first inclusion holds modulo a pluripolar set. The comparison principle for functions in 
$\Ec(X,\omega)$ (see \cite{GZ07}) yields
$$
\int_E MA(u)\leq \int_{\{h<u\}} MA(u)\leq \int_{\{h<u\}} MA(h)\leq \int_{\{h<\p\}} MA(h).
$$
By taking the supremum over all candidates $u$, we get the result.
\end{proof}

The following result says that the inequality in Lemma \ref{lem: the easy inequality for cap formula} is an equality if $E$ is a compact or open subset of $X$.

\begin{thm}\label{formula}
Let $E$ be an open (or compact) subset of $X$ and let 
$h_E:=h^*_{\f,\psi,E}$ be the $(\f,\p)$-extremal function of 
$E$. The  $(\f,\p)$-capacity  of $E$ is given by
$$
\Capip (E) = \int_{\{h_E<\p\}}  MA(h_E).
$$
\end{thm}
\begin{proof}
From Lemma \ref{lem: the easy inequality for cap formula} above we get one inequality. We now prove the opposite one. Set $h:=h_E$. Assume first that $E$ is a compact subset of $X.$ 
Let  $(\p_j)$ be a sequence of continuous $\omega$-psh functions decreasing to $\p$. Denote by $h_j:=h_{\f,\p_j,E}^*$. It is easy to check that  $h_j$ decreases to $h$ and that
$\Capipj (E)$ decreases to $\Capip (E)$. Since $h_j$ is a candidate defining the $(\f,\p_j)$-capacity of $E$, it follows from Proposition \ref{prop: maximal} 
and Lemma \ref{lem: the easy inequality for cap formula} that 
\begin{equation}\label{eq: 10h 47 16 Juin 2013}
\Capipj (E)= \int_{\{h_j<\p_j\}}  MA(h_j)=  \int_{E}  MA(h_j).
\end{equation}
Fix $j_0\in \N.$ Since $h_j\leq h_{j_{0}}$ and $\p\leq \p_j$, for any $j>j_{0}$
$$
\int_{\{h_j<\p_j\}}  MA(h_j) \geq \int_{\{h_{j_0}<\p\}}  MA(h_j).
$$
Fix $\varepsilon >0$ and replacing $\p$ by a continuous function $\tilde{\p}$ such that $\Capa_\omega(\{\tilde{\p}\neq \p\})<\varepsilon$. Arguing as in the proof of Proposition \ref{prop: maximal} we get
\begin{eqnarray*}
\liminf_{j\rightarrow +\infty} \int_{\{h_{j_0}<\p\}}  MA(h_j)\geq \int_{\{h_{j_0}<\p\}}  MA(h).
\end{eqnarray*}
Taking the limit for $j\rightarrow +\infty$ in (\ref{eq: 10h 47 16 Juin 2013}) we get
$$
\Capip (E)\geq \int_{\{h<\p\}}  MA(h).
$$ 
We now assume that $E\subset X$ is an open set. Let $(K_j)$ be a sequence of compact subsets increasing to $E$. Then clearly 
$h_j:=h^*_{\f,\p,K_j}\searrow h$ and $\Capip(K_j)\nearrow \Capip(E) $. We have already proved that $\Capip(K_j)\geq \int_{\{h_j<\p\}}  MA(h_j)$. 
For each fixed $k\in \N$, we have 
$$
\liminf_{j\to+\infty}\int_{\{h_j<\p\}} MA(h_j)\geq \liminf_{j\to+\infty}\int_{\{h_k<\p\}} MA(h_j)\geq \int_{\{h_k<\p\}} MA(h).
$$
Then letting $k\to+\infty$ and using the first part of the proof we get
$$
\liminf_{j\to+\infty} \Capfp(K_j) \geq \int_{\{h<\p\}} MA(h) .
$$
On the other hand, it is clear that 
$\lim_{j\to +\infty} \Capfp (K_j)= \Capfp (E)$, and hence 
$$
\Capfp (E)\geq \int_{\{h<\p\}} MA(h).
$$ 
\end{proof}

Now we want to give a formula for the outer $(\f,\p)$-capacity. Assume that $E$ is a Borel subset of $X$. We introduce an auxiliary function 
\begin{equation}\label{eq: auxiliary function}
\phi:=\phi_{\f,\p,E}=\begin{cases} \frac{-h_{\f,\p,E}^*+\p}{-\f+\p} \quad {\rm if}\ \f>-\infty \\ \;0 \quad\quad\quad {\rm if}\  \f=-\infty  \end{cases}  .
\end{equation}
Observe that $\phi$ is a quasicontinuous function, $0\leq \phi\leq 1$ and $\phi=1$ quasi everywhere on $E$.
\begin{thm}\label{thm: cap formula varphi psi}
Let  $E\subset X$ be a Borel subset and denote by $h_E:=h^*_{\f,\p,E}$ the $(\f,\p)$-extremal function of $E$. Then 
$$
\Capfp^*(E)=\int_{\{h_E<\p\}} \MA(h_E)=\int_{X} \left( \frac{-h_E+\p}{-\f+\p} \right)\, \MA(h_{E}).
$$
\end{thm}
To prove Theorem \ref{thm: cap formula varphi psi}  we need the following results.

\begin{lem}\label{lem1}
Let $(u_j)$ be a bounded monotone sequence of quasi-continuous functions converging to $u$.  Let $\chi$ be a convex weight and $\{\varphi_j\}\subset \Ec_{\chi}(X,\omega)$ be a monotone sequence converging to $\varphi\in  \Ec_{\chi}(X,\omega)$. Then
$$
\int_{X} u_j\, \MA(\varphi_j) \xrightarrow[j\rightarrow +\infty]{}  \int_{X} u\, \MA(\varphi). 
$$
\end{lem}

\begin{proof}
Fix $\varepsilon>0.$ Let $U$ be an open subset of $X$ with $\Capa_{\omega}(U)<\varepsilon$ and $v_j, v$ be continuous functions on $X$ such that $v_j\equiv u_j$ and $v\equiv u$ on $K:=X\setminus U.$ By Theorem \ref{thm: comparison of Capacity} (and by letting $\varepsilon\to 0$) it suffices to prove that 
$$
\int_X v_j\, \MA(\varphi_j) \xrightarrow[j\rightarrow +\infty]{}  \int_{X} v\, \MA(\varphi)\, .
$$
From Dini's theorem $v_j$ converges uniformly to $v$ on $K$. Thus, using again Theorem \ref{thm: comparison of Capacity}, the problem reduces to proving that 
$$
\int_X v\, \MA(\varphi_j) \xrightarrow[j\rightarrow +\infty]{}  \int_{X} v\, \MA(\varphi)\, .
$$
But the latter obviously follows since $v$ is continuous on $X$. The proof is thus complete.
\end{proof}

\begin{prop}\label{prop: formula compact open varphi psi}
Let $E$ be a compact or open subset of $X$ and let 
$h_E:=h^*_{\f,\p,E}$ denote the $(\f,\p)$-extremal function of $E$. Then
$$
\Capip(E)=\int_{\{h_E<\p\}}\MA(h_E)=\int_X \left( \frac{-h_E+\p}{-\f+\p} \right) \, \MA(h_E).
$$
\end{prop}

\begin{proof}
The first equality has been proved in Theorem \ref{formula}. 
Set $h:=h_E$ and $\phi:=\phi_{\f,\p,E}=\frac{-h_E+\p}{-\f+\p}$. Observe that
$\{h<\p\}=\{\phi>0\}$ modulo a pluripolar set and $\phi\leq 1.$ Thus 
$$
\int_{\{h<\p\}}\MA(h)\geq \int_X \phi \, \MA(h).
$$ 
Assume that $E$ is compact. By Proposition \ref{prop: maximal} and Theorem \ref{formula} we have 
$$
\Capip(E)=\int_{E}\MA(h).
$$ 
Since $\phi = 1$ quasi everywhere on $E$ we obtain
$$
\int_{E}\MA(h)\leq \int_X \phi\, \MA(h).
$$ 
We assume now that $E\subset X$ is an open subset. Let $(K_j)$ be a sequence of compact subsets increasing to $E$.  Then
$$
\Capip(E)=\lim _{j\rightarrow +\infty} \Capip(K_j)
=\lim _{j\rightarrow +\infty} \int_X \phi_j\, \MA(h_j) ,
$$
where $h_j:= h_{\f,\p, K_j}^*$ and $\phi_j:= \phi_{\f,\p,K_j}$ is defined by (\ref{eq: auxiliary function}). Since $\phi_j$ is quasicontinuous  for any $j$ and $\phi_j\searrow \phi$, the conclusion follows from Lemma \ref{lem1}.
 \end{proof}

\begin{lem}\label{lem3}
Let $u,v$ be $\omega$-plurisubharmonic functions. Let $G\subset X$ be an open subset. Set $E=\{u<v\}\cap G$ and $h_E:=h_{\f,\p,E}^*$. Then
$$
\Capfp^*(E)=\Capfp(E)=\int_{\{h_E<\p\}}\MA(h_E)=\int_X \left( \frac{-h_E+\p}{-\f+\p} \right)\, \MA(h_E).
$$
\end{lem}
\begin{proof}
We start showing the first identity. First, just by definition $\Capip^*(E)\geq\Capip(E)$. Fix $\varepsilon >0$. There exists a function $\tilde{v}\in \Cc(X)$ such that 
$$
\Capo(\{\tilde{v}\neq v\})< \varepsilon.
$$
Clearly $E\subset \left(\{u<\tilde{v}\}\cap G \right) \cup \{\tilde{v}\neq v\}$ and so, applying Theorem \ref{thm: comparison of Capacity} we get
\begin{eqnarray*}
\Capip^*(E)& \leq & \Capip(\{u<\tilde{v}\}\cap G) + F(\varepsilon)\\
& \leq & \Capip(E) + 2F(\varepsilon),
\end{eqnarray*}
where $F(\varepsilon)\to 0$ as $\varepsilon\to 0.$ 
Taking the limit as $\varepsilon\rightarrow 0$ we arrive at the first conclusion. 

Let now $\{K_j\}$ be a sequence of compact sets increasing to $G$ and $\{u_j\}$
be a sequence of continuous functions decreasing to $u$. 
Then $E_j= \{u_j+1/j\leq v\}\cap K_j$ is  compact  and $E_j\nearrow E$. Set 
$$
h:=h_{\f,\p,E}, \, \phi:=\frac{-h_E+\p}{-\f+\p}, \, h_j:=h_{\f,\p,E_j}^*, \, \phi_j:=\frac{-h_{E_j}+\p}{-\f+\p}.
$$ Observe that $h_j\searrow h$ and $\phi_j\searrow \phi$. 
By Proposition \ref{prop: formula compact open varphi psi} and Lemma \ref{lem1} we have
\begin{eqnarray*}
\Capip(E)&=& \lim _{j\rightarrow +\infty} \Capip(E_j)\\
&=&  \lim _{j\rightarrow +\infty}  \int_X \phi_j\, \MA(h_j)\\
&=& \int_X \phi\, \MA(h) \leq \int_{\{h<\p\}} \MA(h).
\end{eqnarray*}
Furthermore, for each fixed $k\in \N$, using Theorem \ref{thm: comparison of Capacity} we can argue as above to get
$$
\liminf_{j\to+\infty}\int_{\{h_j<\p\}}\MA(h_j)\geq \liminf_{j\to+\infty}\int_{\{h_k<\p\}}\MA(h_j) \geq \int_{\{h_k<\p\}}\MA(h) .
$$  
Letting $k\to+\infty$ and using Proposition \ref{prop: formula compact open varphi psi} again we get 
$$
\Capip(E)\geq \int_{\{h<\p\}}\MA(h) ,
$$
which completes the proof.
\end{proof}
We are now ready to prove Theorem \ref{thm: cap formula varphi psi}.
\begin{proof}
As usual, for simplicity, set $h:=h_E$. By definition of the outer capacity there is  a sequence  $(O_j)$ of open sets  decreasing to $E$ such that $\Capip^*(E)=\lim _{j\rightarrow +\infty}\Capip(O_j)$. Furthermore by Choquet's lemma there exists a sequence $(u_j)$ of $\omega$-psh functions such that 
$u_j\equiv \f$ quasi everywhere on $E$, $u_j\leq \p$ on $X$ and $u_j\nearrow h$. Since $\Capip^*$ vanishes on pluripolar sets (see Theorem \ref{thm: pluripolar}) we can assume that $u_j\equiv \f$ on $E$.  
For any $j$, we set $E_j= O_j\cap \{u_j<\f+1/j\}$ and $h_j:=h_{\f,\p,E_j}^*$. Then $(E_j)$ is a decreasing 
sequence of open subsets such that $E\subset E_j\subset O_j$ 
and $u_j-1/j\leq h_j\leq h$, thus $h_j\nearrow h$. 
Clearly $\Capip^*(E)=\lim _{j\rightarrow +\infty}\Capip(E_j)$. 
By Lemma \ref{lem3} and Lemma \ref{lem1} we get
$$
\lim _{j\rightarrow +\infty}\Capip^*(E_j)=\lim _{j\rightarrow +\infty}\Capip(E_j)= \lim _{j\rightarrow +\infty}\int_X \phi_j\, \MA(h_j)= \int_X \phi\, \MA(h),
$$
where $\phi_j:=\phi_{\f,\p,E_j}$ is defined by (\ref{eq: auxiliary function}).
\end{proof}

\begin{cor}
Let $K\subset X$ be a compact set and $(K_j)$ a sequence of compact subsets decreasing to $K$. Then
\begin{itemize}
\item[(i)] $\Capip^*(K)=\Capip(K) = \lim_{j\to+\infty}\Capip (K_j)$,
\item[(ii)] $h_{\f,\p,K_j}^*\nearrow h_{\f,\p,K}^*$.
\end{itemize}
\end{cor}
\begin{proof}
The first equality in statement (i) comes straightforward from Theorem \ref{formula} and Theorem \ref{thm: cap formula varphi psi}. The second one follows from (ii) and Theorem \ref{thm: cap formula varphi psi}. It remains to  prove (ii). Since $(K_j)$ 
decreases to $K$, $h_j:=h_{\f,\p,K_j}^*$ increases to some 
$h_{\infty}\in \Ec(X,\omega)$. Clearly $h_{\infty}\leq h$. 
Thus we need to prove that $h_{\infty}\geq h$. 
Since $\{h_{\infty}<h\}\subset \{h_{\infty}<\p\}\setminus K$ 
modulo a pluripolar set,
$$ 
\int_{\{h_{\infty}<h\}} \MA(h_{\infty})\leq \int_{\{h_{\infty}<\psi\}\setminus K} \MA(h_{\infty}).
$$
From Proposition \ref{prop: maximal} we know that 
$$
\int_{\{h_j<\p\}\setminus K_j} \MA(h_j) =0,\, \forall j\in \N .
$$
Fix $\varepsilon >0$ and let $\p_{\varepsilon}\in \Cc(X)$ such that 
$\Capo(\{\p_{\varepsilon}\neq\p\})<\varepsilon$.  Then for each fixed 
$k\in \N$, we have
\begin{eqnarray*}
\int_{\{h_{\infty}<\p\}\setminus K_k} \MA(h_{\infty}) &\leq & 
\int_{\{h_{\infty}<\p_{\varepsilon}\}\setminus K_k} \MA(h_{\infty})+ F(\varepsilon)\\
&\leq & \liminf_{j\to+\infty}\int_{\{h_{\infty}<\p_{\varepsilon}\}\setminus K_k} \MA(h_j)+ F(\varepsilon)\\
&\leq & \liminf_{j\to+\infty}\int_{\{h_{\infty}<\p\}\setminus K_k} \MA(h_j)+ 2 F(\varepsilon)\\
&\leq & \liminf_{j\to+\infty}\int_{\{h_j<\p\}\setminus K_k} \MA(h_j)+
 2 F(\varepsilon)\\
&=& 2 F(\varepsilon) ,
\end{eqnarray*}
where $F(\varepsilon)\to 0$ as $\varepsilon\to 0$ thanks to Theorem \ref{thm: comparison of Capacity}. 
Thus, letting $\varepsilon\to 0$ then $k\to+\infty$ and using the domination principle below (Proposition \ref{prop: domination principle}) we can conclude that $h_{\infty}\geq h$.
\end{proof}

\subsection{Proof of Theorem A} Let us briefly resume the proof of Theorem A. Statements (i) and (ii) have been proved in Theorem \ref{thm: cap formula varphi psi} and Theorem \ref{thm: comparison of Capacity} respectively. One direction of the last staement has been proved in Theorem \ref{thm: pluripolar}. Now, if 
$E$ is a Borel subset of $X$ such that $\Capfp^*(E)=0$ then it follows from 
Theorem \ref{thm: cap formula varphi psi} that 
$$
\int_{\{h_{\f,\p,E}^*<\p\}} \MA(h_{\f,\p,E}^*) =0.
$$
We then can apply the domination principle (see \cite{BL12} or Proposition \ref{prop: domination principle} below for a proof)  to conclude.

\section{Another proof of the Domination Principle}
The following domination principle  was proved by Dinew using his uniqueness result \cite{Di09}, \cite{BL12}. As an application of the $(\f,\psi)$-Capacity we propose here an alternative proof.

\begin{prop}\label{prop: domination principle}
If $u, v\in \Ec(X,\omega)$ such that $u\leq v$ $MA(v)$-almost everywhere then $u\leq v$ on $X.$
\end{prop}
\begin{proof}
We first claim that for every $\f\in \Ec(X,\omega)$ such that  $0\leq  \f-u \leq C$ for some constant $C>0$ and for any $s>0$ one has 
$$
\int_{\{v<u-s\}}MA(\f)=0.
$$
Indeed, fix $s>0$ and let $\f$ be such a function. Let $C>0$ be a constant such that $\f-u\leq C$ on $X.$ Choose $\delta\in (0,1)$ such that $\delta C<s.$ Now, by using the comparison principle and the fact that $0\leq \f-u\leq C$ we get
\begin{eqnarray*}
\delta^n\int_{\{v<u-s \}} MA(\f)&= & \int_{\{v<u-s\}}(\delta\omega+dd^c \delta \f)^n \\
&\leq &\int_{\{v<\delta\f+(1-\delta) u-s\}} MA\left (\delta \f +(1-\delta) u\right)\\
&\leq & \int_{\{v<\delta\f+(1-\delta) u-s\}} MA(v)\\
&\leq & \int_{\{v<u\}} MA(v)=0 .
\end{eqnarray*}
Thus, the claim is proved.  Now for each $t>0$ let $h_t$ denote the 
$(u,0)$-extremal function of the open set $G_t:=\{u<-t\}.$ It is clear that for every $t>0,$ $h_t\in \Ec(X,\omega)$ and $\sup_X ( h_t-u)<+\infty.$ The previous step yields 
$$
\int_{\{v<u-s\}}MA(h_t)=0, \ \forall s>0.
$$
Fix $\varepsilon >0$. Let $\tilde{u}$ be a continuous function on $X$ such that $\Capo(\{u\neq \tilde{u}\})<\varepsilon$. Since $h_t$ increases to $0$ (see Lemma \ref{lem: pluripolar} below),  we infer that 
$$
\int_{\{v<\tilde{u}-s\}}\omega^n\leq \liminf_{t\to +\infty} \int_{\{v<u-s\}} \MA(h_t)+ \Capa_{u,0}(\{u\neq \tilde{u}\}).
$$ 
Repeating this argument we get 
$$
\int_{\{v<u-s\}}\omega^n \leq \varepsilon +\Capa_{u,0}(\{u\neq \tilde{u}\}).
$$
Letting $\varepsilon \to 0$ and using Theorem \ref{thm: comparison of Capacity} we get ${\rm Vol}(\{v<u-s\})=0$, for any $s>0$ which implies
 that $u\leq v$ on $X$ as desired.
\end{proof}

\begin{lem}\label{lem: pluripolar}
Let $v\in \psh(X,\omega).$ For each $t>0$, set $G_t:=\{v < -t \}$. Denote by $h_t$ the $(\f,0)$-extremal function of $G_t$. Then $h_t$ increases quasi everywhere on $X$ to $0$ when $t$ increases to $+\infty$.
\end{lem}
\begin{proof}
We know that $h_t$ increases quasi everywhere to $h\in \Ec(X,\omega)$ and that  $h\leq 0$. 
By Theorem \ref{thm: pluripolar} (up to consider $-(-v)^p$ with $p\in(0,1)$ instead of $v$), we get
$$
\lim_{t\to+\infty}\Capa_{\f,0} (G_t)= 0 .
$$
It follows from Theorem \ref{formula} that for each $t>0$,
$$
\int_{\{h<0\}}MA(h_t)\leq \int_{\{h_t<0\}}MA(h_t)=\Capa_{\f,0}  (G_t).
$$
We thus get
$$
\int_{\{h<0\}}MA(h)\leq \liminf_{t\to+\infty}\int_{\{h<0\}}MA(h_t)=0.
$$
The comparison principle yields ${\rm Vol}( \{h<0\})=0$ which completes the proof.
\end{proof}

\begin{rem}
Lemma \ref{lem: pluripolar} is stated and proved in the case $\p\equiv 0$. Observe that it also holds for any $\p\in \EcX$ such that $\f<\p$. To see this we can follow the same arguments of above but for the final step where we get $\p\leq h \;\, \MA(h)$-almost everywhere. We then conclude using the domination principle.
\end{rem}

\medskip

\section{Applications to Complex Monge-Amp\`ere equations }

In this section (in the same spirit of \cite{DL1}) we prove Theorem B by using  $\Capa_\p:=\Capa_{\p-1,\p}$. Let us recall the setting. Let $X$ be a compact K\"{a}hler manifold of dimension $n$ and let $\omega$ be a K\"ahler form on $X$. Let $D$ be an arbitrary  divisor on $X$. Consider the complex Monge-Amp\`ere equations 
\begin{equation}\label{eq: 2}
 (\omega+dd^c \varphi)^n=e^{\lambda\f}f\omega^n, \ \lambda\in \R.
\end{equation}
We say that $f$  satisfies Condition $\Hc_f$ if 
$$
f=e^{\p^+-\p^-}, \ \ \p^{\pm} \ {\rm are \ quasi \ psh \ functions\ on }\ X\ , \ \p^-\in L^{\infty}_{\rm loc}(X\setminus D).
$$

We have already treated the case when $\lambda=0$ in \cite{DL1}. 
If $\lambda>0$ and $f$ is integrable then the same arguments can be applied. More precisely, 
$\Cc^0$-estimates follow from comparison principle while the $\Cc^2$ estimate 
follows exactly the same way as in \cite{DL1}. 

The case when $\lambda<0$ is known to be much more difficult. We need the
following observation where we make use of the generalized capacity $\Capis$:
\begin{lem} \label{lem: uniform estimate}
Let $\f\in \EcX$ be normalized by $\sup_X \f=0$. Assume that there exist a positive constant $A$ and $\psi\in \psh(X,\omega/2)$ such that
$\MA(\f) \leq e^{-A\p} \omega^n$. Then there exists $C>0$ depending only on
$\int_X e^{-2A\f} \omega^n$ such that 
$$
\f\geq \psi -C.
$$
\end{lem}

Observe that for all $A>0$ and $\f\in \EcX$, $e^{-A\f}\omega^n\in L^1(X)$ as follows from Skoda integrability theorem \cite{Sk72}, \cite{Zer01}, since functions in $\EcX$ have zero Lelong number at all points \cite{GZ07}. 

\begin{proof}
Set
$$
H(t)=\left[\Capa_{\psi}(\{\varphi<\psi-t\})\right]^{1/n},\ t>0.
$$
Observe that $H(t)$ is right-continuous and $H(+\infty)=0$ (see \cite[Lemma 2.6]{DL1}). It follows from \cite[Lemma 2.7]{DL1} that 
$\Capo\leq 2^n \Capa_{\psi}$. By a strong volume-capacity domination in \cite{GZ05} we also have 
$$
\vol_{\omega} \leq \exp{\left(\frac{-C_1}{\Capo^{1/n}}\right)},
$$
where $C_1$ depends only on $(X,\omega)$. 
Thus using \cite[Proposition 2.8]{DL1} and the assumption on the measure $\MA(\f)$, we get
\begin{eqnarray*}
s^n \Capis(\{\varphi<\psi-t-s\})&\leq & \int_{\{\varphi<\p-t\}}\MA(\varphi) \\
&\leq & \int_{\{\varphi<\psi-t\}}e^{-A\varphi} e^{A\psi} \MA(\varphi) \\
&\leq & \left[\int_X e^{-2A\varphi} \omega^n\right]^{1/2}
\left[\int_{\{\varphi<\psi-t\}} \omega^n \right]^{1/2}\\
&\leq & C_2 \left[\Capis(\{\varphi<\psi-t\})\right]^2,
\end{eqnarray*}
where $C_2$ depends on $\int_X e^{-2A \f}\omega^n$. We then get
$$ 
s H(t+s)\leq C_2^{1/n} H(t)^2,  \ \forall t>0, \forall s\in [0,1]. 
$$
Then by \cite[Lemma 2.4]{EGZ09} we get $\varphi \geq \psi-C_3$,
where $C_3$ only depends on $\int_X e^{-2A \f}\omega^n$.
\end{proof}

\smallskip
Now, we are ready to prove Theorem B.
\subsection{Proof of Theorem B} It suffices to treat the case when $\lambda=-1$.
Since $f$ satisfies Condition $\Hc_f$ we can write $\log f=\psi^+-\psi^-$, where $\psi^{\pm}$ are qpsh functions on $X$, $\psi^-$ is locally bounded on $X\setminus D$ and there exists a uniform constant $C>0$ such that 
$$
dd^c \psi^{\pm}\geq -C \omega, \; \sup_X \psi^+ \leq C.
$$
We apply the smoothing kernel $\rho_\varepsilon$ in Demailly regularization theorem \cite{Dem92} to the functions $\f$ and 
$\psi^{\pm}$. 
For $\vep$ small enough, we get
$$ 
dd^c \rho_\varepsilon(\f+\p^-) \geq -C_1 \omega,\;\; 
dd^c \rho_\vep(\psi^+)\geq -C_1 \omega, \;\; 
\sup_X \rho_\vep(\psi^+) \leq C_1,
$$
where $C_1$ depends on $C$ and the Lelong numbers of the currents $C\omega+dd^c \psi^{\pm}$.
By the classical result of Yau \cite{Y78}, for each $\varepsilon$, there exists a unique smooth $\omega$-psh function $\phi_{\vep}$ satisfying
$$
\MA(\phi_\varepsilon)=e^{c_{\varepsilon}+\rho_\vep(\psi^+)-\rho_{\vep}(\f+\psi^-)}  \omega^n=g_\varepsilon \omega^n, \ \ \sup_X\phi_{\vep}=0,
$$
where $c_\varepsilon$ is a normalization constant such that 
$$
\int_X g_\varepsilon \omega^n=\int_X e^{-\f}f \omega^n =\int_X \omega^n.
$$
Since by Jensen's inequality $e^{\rho_\varepsilon(-\f+\log f)}\leq \rho_\varepsilon (e^{-\f+\log f})$ and $e^{\rho_\varepsilon(-\f+\log f)}$ converges point-wise to $e^{-\f}f$ on $X$, it follows from the general Lebesgue dominated convergence theorem that $e^{\rho_\varepsilon(-\f+\log f)}$ converges to $e^{-\f}f$ in $L^1(X)$ when 
$\vep\downarrow 0$. This means that $c_\varepsilon$ converges to zero when $\varepsilon\rightarrow 0$. It then follows from \cite[Lemma 3.4]{DL1} that $\phi_{\vep}$ converges in $L^1(X)$ to $\f-\sup_X \f$. 
We now apply the $\Cc^2$ estimate in \cite[Theorem 3.2]{DL1} to get
$$
n+\Delta \phi_{\vep} \leq C_3 e^{-2\rho_{\vep}(\f+\p^-)} \leq 
C_4 e^{-2(\f+\p^-)},
$$
where $C_3, C_4$ are uniform constants (do not depend on $\vep$).
Now, we need to bound $\f$ from below. By the assumption on $f$ we have 
$$
\MA(\f) = e^{\p^+-(\f+\p^-)} \omega^n \leq e^{-(\f+\p^--C)}\omega^n. 
$$
Consider $\p:=\frac{1}{2C+2} (\f+\p^-)$. Since this function belongs 
to $\psh(X,\omega/2)$ we can apply Lemma \ref{lem: uniform estimate} to get 
$$
\f-\sup_X \f \geq \p - C_5. 
$$
This gives $\f\geq C_6 \p^- -C_7$. Applying again this argument to 
$\phi_{\vep}$ and noting that $c_{\vep}$ converges to $0$, and hence under control, we get
$$
\phi_{\vep} \geq \rho_{\vep}(\f+\p^-)  - C_8 \geq C_9 \p^- -C_{10}.
$$
We can now conclude using the same arguments in \cite[Section 3.3]{DL1}. 

\subsection{(Non) Existence of solutions}
In the previous subsection, no regularity assumption on $D$ has been done. We now discuss about the existence of solutions in concrete examples, assuming more information on $D,f$.

Let $D= \sum_{j=1}^{N} D_j$ be a simple normal crossing divisor on $X$. Reacall that "simple normal crossing" means that around each intersection point of $k$ components $D_{j_1},...,D_{j_k}$ ($k\leq N$), we can find complex coordinates $z_1,...,z_n$ such that 
for each $l=1,...,k$ the hypersurface $D_{j_l}$ is locally given by $z_l=0$. 

For each $j$, let $L_j$ be the holomorphic line bundle defined by $D_j$. Let $s_j$ be a holomorphic section of $L_j$ defining $D_j$, i.e 
$D_j=\{s_j=0\}$. We fix  a hermitian metric $h_j$ on $L_j$ such that 
$\vert s_j\vert:= \vert s_j\vert_{h_j}\leq 1/e$.

We assume that $f$ has the following particular form:
\begin{equation}\label{eq: cond f}
f =  \frac{h}{\prod_{j=1}^N \vert s_j\vert^2 (-\log \vert s_j\vert)^{1+\alpha}}, \ \alpha>0,
\end{equation}
where $h$ is a bounded function: $0<1/B \leq h \leq B, \; B>0$. \\
\textbf{In this subsection we always assume that $\lambda<0$.} 

\begin{prop}\label{prop: non existence}
Assume that $f$ satisfies (\ref{eq: cond f}) with $0<\alpha\leq 1$. Then there is no solution in $\EcX$ to equation 
$$
(\omega+dd^c \varphi)^n=e^{\lambda\f} f \omega^n.
$$ 
\end{prop}
\begin{proof}
We can assume (up to normalization) that $\lambda=-1$. Then observe that if there exists $\f\in \Ec(X,\omega)$ such that 
$
(\omega+dd^c \varphi)^n=e^{-\f}\mu,
$
where $\mu$ is a positive measure, then we can find $A>0$ such that
$$
\mu\leq A \left(\omega+dd^c u\right)^n, 
$$
where $u:=e^{(\f-\sup_X \f)/n}$ is a bounded $\omega$-psh function. Indeed, $u$ is a $\omega$-psh function and
$$\omega+dd^c u \geq \omega + \frac{u}{n}dd^c \f\geq \frac{u}{n} (\omega+dd^c \f) \geq 0.$$
This coupled with \cite[Proposition 4.4 and 4.5]{DL1} yields the conclusion.
\end{proof}

The above analysis shows that there is no solution if the density has singularities of Poincar\' e type or worse. We next show that when $f$
is less singular than the Poincar\'e type density (i.e. $\alpha>1$), equation
(\ref{eq: 2}) has a bounded solution provided $\lambda=-\vep$ with 
$\vep>0$ very small. We say that $f$ satisfies Condition $\Sc(B,\alpha)$ for some $B>0$, $\alpha>0$ if 
$$
f \leq  \frac{B}{\prod_{j=1}^N \vert s_j\vert^2 (-\log \vert s_j\vert)^{1+\alpha}}.
$$

\begin{thm}
Assume that $f$ satisfies Condition $\Sc(B,\alpha)$ with $\alpha>1$. We also normalize $f$ so that 
$\int_X f\omega^n=\int_X \omega^n$. Then for $\lambda=-\vep$ with 
$\vep>0$ small enough
depending only on $C,\alpha,\omega$,
there exists a bounded solution $\f$ to  (\ref{eq: 2}).\\
The solution 
is automatically continuous on $X$. In particular, it is also smooth on $X\setminus D$ if $f$ is smooth there.
\end{thm}
\begin{proof}
The last statement follows easily from our previous analysis. Let us prove the existence. We use the Schauder Fixed Point Theorem. Let $C=C(2B,\alpha)$ be the constant in Lemma \ref{lem: uniform bound} below. Choose $\vep>0$ very small such that $e^{\vep C}\leq 2$.
Consider the following compact convex set in $L^1(X)$:
$$
\Cc:= \{u \in \psh(X,\omega)\ \setdef \ -C\leq u\leq 0 \}.
$$ 
Let $\p\in \Cc$ and $c_{\p}$ be a constant such that 
$$
\int_X e^{-\vep \p +c_{\p}} f\omega^n = \int_X \omega^n.
$$
Since $-C\leq \p\leq 0$, it is clear that $-C\vep\leq c_{\p}\leq 0$. Let $\f$ be the unique bounded $\omega$-psh function such that $\sup_X \f=0$ and 
$$
(\omega+dd^c \f)^n= e^{-\vep \p +c_{\p}}f\omega^n.
$$
The density on the right-hand side satisfies Condition 
$\Sc(B,\alpha)$ since $c_{\p}\leq 0$ and since $e^{\vep C}\leq 2$.
We thus get from Lemma \ref{lem: uniform bound} below that $\f\geq -C$. Thus we have defined a mapping from $\Cc$ to itseft
$$
\Phi : \Cc \rightarrow \Cc, \ \ \ \Phi(\p):=\f.
$$
Let us prove that $\Phi$ is continuous on $\Cc$. Let $\p_j$ be a sequence in $\Cc$ which converges to $\p$ in $L^1(X)$. Denote by
$$
c_j:=c_{\p_j}, \ \ c:=c_{\p}, \ \Phi(\p_j)=\f_j,\ \Phi(\p)=\f.
$$
It is enough to prove that any cluster point of the sequence $(\f_j)$
is equal to $\f$. Therefore, we can assume that $\f_j$ converges to 
$\f_0$ in $L^1(X)$ and up to extracting a subsequence that 
$\p_j$ converges almost everywhere to $\p$ on $X$ and also that 
$c_j$ converges to $c_0\in [-C\vep,0]$. Since $e^{-\vep \p_j +c_j}f$ converges in $L^1(X)$ to $e^{-\vep \p +c_0}f$ in $L^1(X)$ and almost everywhere, it follows from \cite[Lemma 3.4]{DL1} that 
$$
(\omega+dd^c \f_0)^n = e^{-\vep \p +c_0}f\omega^n.
$$
It is clear that $c_0=c$ and it follows from Hartogs' lemma that
$\sup_X \f_0=0$. Thus $\f_0=\f$. This concludes the continuity of 
$\Phi$.
 
Now, since $\Cc$ is compact and convex in $L^1(X)$ and since $\Phi$
is continuous on $\Cc$, by Schauder Fixed Point Theorem there exists 
a fixed point of $\Phi$, say $\f$. Then $\f-c_{\f}/{\vep}$ is the desired solution.
\end{proof}

We refer the reader to \cite[Section 4.2]{DL1} for the proof of the following lemma.

\begin{lem}\label{lem: uniform bound}
Assume that  $f$ satisfies Condition $\Sc(B,\alpha)$ with $\alpha>1$, $B>0$. Let $\f\in \Ec(X,\omega)$ be the unique function such that 
$$
(\omega + dd^c \f)^n =f\omega^n, \ \sup_X \f=0.  
$$
Then $\f\geq -C$ with $C=C(B,\alpha)>0$.
\end{lem}

\subsection{Proof of Theorem C} Assume that  $\f\in \EcX$ satisfies 
$$
(\omega+dd^c\f)^n=e^{\lambda\f}f\omega^n, \lambda>0.
$$
Up to rescaling $\omega$ it suffices to treat the case when 
$\lambda=1$. The proof of Theorem C is quite similar to that of Theorem B. The difference here is that $f$ is not integrable. For convenience of the reader we rewrite the arguments here. Since $f$ satisfies Condition $\Hc_f$ we can write $\log f=\psi^+-\psi^-$, where $\psi^{\pm}$ are qpsh functions on $X$, $\psi^-$ is locally bounded on $X\setminus D$ and there exists a uniform constant $C>0$ such that 
$$
dd^c \psi^{\pm}\geq -C \omega, \; \sup_X \psi^+ \leq C.
$$
We apply the smoothing kernel $\rho_\varepsilon$ in Demailly regularization theorem \cite{Dem92} to the functions $\f$ and 
$\psi^{\pm}$. 
For $\vep$ small enough, we get
$$ 
dd^c \rho_\varepsilon(\p^-) \geq -C_1 \omega,\;\; 
dd^c \rho_\vep(\f+\psi^+)\geq -C_1 \omega, \;\; 
\sup_X \rho_\vep(\f+\psi^+) \leq C_1,
$$
where $C_1$ depends on $C$,  the Lelong numbers of the currents $C\omega+dd^c \psi^{\pm}$ and $\sup_X \f$.
By the classical result of Yau \cite{Y78}, for each $\varepsilon$, there exists a unique smooth $\omega$-psh function $\phi_{\vep}$ satisfying
$$
\MA(\phi_\varepsilon)=e^{c_{\varepsilon}+\rho_\vep(\f+\psi^+)-\rho_{\vep}(\psi^-)}  \omega^n=g_\varepsilon \omega^n, \ \ \sup_X\phi_{\vep}=0,
$$
where $c_\varepsilon$ is a normalization constant such that 
$$
\int_X g_\varepsilon \omega^n=\int_X e^{\f}f \omega^n =\int_X \omega^n.
$$
Since by Jensen's inequality $e^{\rho_\varepsilon(\f+\log f)}\leq \rho_\varepsilon (e^{\f+\log f})$ and 
$e^{\rho_\varepsilon(\f+\log f)}$ converges point-wise to 
$e^{\f}f$ on $X$, it follows from the general Lebesgue dominated convergence theorem that $e^{\rho_\varepsilon(\f+\log f)}$ converges to $e^{\f}f$ in $L^1(X)$ when 
$\vep\downarrow 0$. This means that $c_\varepsilon$ converges to zero when $\varepsilon\rightarrow 0$. It then follows from Lemma 3.4 in \cite{DL1} that $\phi_{\vep}$ converges in $L^1(X)$ to $\f-\sup_X \f$. 
We now apply the $\Cc^2$ estimate in Theorem 3.2 in \cite{DL1} to get
$$
n+\Delta \phi_{\vep} \leq C_3 e^{-2\rho_{\vep}(\p^-)} \leq 
C_4 e^{-2\p^-},
$$
where $C_3, C_4$ are uniform constants (do not depend on $\vep$).
Now, we need to bound $\f$ from below. By the assumption on $f$ we have 
$$
\MA(\f) = e^{\f+\p^+-\p^-} \omega^n \leq e^{-(\p^--C_1)}\omega^n. 
$$
Consider $\p:=\frac{1}{2C}\p^-$. Since this function belongs 
to $\psh(X,\omega/2)$ we can apply Lemma \ref{lem: uniform estimate} to get 
$$
\f-\sup_X \f \geq \p - C_5. 
$$
Now the remaining part of the proof follows by exactly
the same way as we have done in \cite[Section 3.3]{DL1}. 

\subsection{Non Integrable densities}
When $0\leq f\notin L^1(X)$ it is not clear that we can find a solution $\f\in \EcX$ of equation 
$$
(\omega+dd^c \f)^n =e^{\f}f\omega^n.
$$
We show in the following that it suffices to find a subsolution. Another similar result has been proved by Berman and Guenancia in \cite{BG13} using the variational approach. We provide here a simple proof using 
our generalized Monge-Amp\`ere capacities. 

\begin{thm}\label{thm: subsol}
Let $0\leq f$ be a measurable function such that $\int_X f\omega^n=+\infty$. If there exists $u\in \EcX$ such that $\MA(u)\geq e^{u} f\omega^n$ then there is a unique $\f\in \EcX$ such that 
$$
\MA(\f) =e^{\f}f\omega^n.
$$
\end{thm}

\begin{proof}
The uniqueness follows easily from the comparison principle. Indeed, one can find a proof in \cite[Proposition 3.1]{BG13}. We now establish the existence. For each $j\in \N$ we can find $\f_j \in \psh(X,\omega) \cap L^{\infty}(X)$ such that 
$$
(\omega+dd^c \f_j)^n = e^{\f_j} \min(f,j) \omega^n. 
$$
It follows from the comparison principle that $\f_j$ is non-increasing 
and $\f_j \geq u$. Then $\f_j\downarrow \f\in \EcX$ and by continuity
of the complex Monge-Amp\`ere operator along decreasing sequence in 
$\EcX$ we get 
$$
\MA(\f) = e^{\f}f\omega^n.
$$
Indeed, since $\MA(\f_j)$ converges weakly to $\MA(\f)$, from Fatou's lemma we get 
$$
\MA(\f) \geq e^{\f} f\omega^n
$$
in the sense of positive Borel measures. To get the reverse inequality
we need to show that the right-hand side has full mass, i.e. 
$$
\int_X e^{\f} f\omega^n =\int_X \omega^n. 
$$
Fix $\vep >0$. Since $\f$ is $\omega$-psh, in particular quasi-continuous, we find $U$ an open subset of $X$ such that
$\Capo(U)<\vep$ and $\f$ is continuous on $K:=X\setminus U$. Then $\f$ is bounded on $K$ and hence $f$ must be integrable on $K$.
We thus can apply the Lebesgue Dominated Convergence Theorem on $K$ to get 
$$
\lim_{j\to +\infty}\int_K \MA(\f_j) = \lim_{j\to +\infty} \int_K e^{\f_j} \min(f,j) \omega^n = \int_K e^{\f}f\omega^n.
$$
We can assume that $\f_j \leq 0$. It follows from Theorem \ref{thm: comparison of Capacity} that 
$$
\int_U \MA(\f_j) \leq \Capa_{u,0}(U) \leq F(\vep) \rightarrow 0\ \ {\rm  as}\ \vep \downarrow 0.
$$
This implies that 
\begin{eqnarray*}
\int_X e^{\f}f\omega^n &\geq & \int_K e^{\f}f\omega^n=\lim_{j\to +\infty}\int_K \MA(\f_j)\\
&=& \int_X \MA(\f_j) - \lim_{j\to+\infty} 
\int_U \MA(\f_j) \\
&\geq &\int_X \omega^n -F(\vep).
\end{eqnarray*}
By letting $\vep\to 0$ we get
$\int_X e^{\f}f\omega^n =\int_X \omega^n$, which completes  the proof.
\end{proof}

\begin{rem}
Theorem \ref{thm: subsol} also holds if $\omega$ is merely semipositive and big. 
\end{rem}

\begin{ex}\label{ex: subsol}
Let $D=\sum_{j=1}^N D_j$ be a simple normal crossing divisor on $X$. Assume that the
$D_j$ are defined by $s_j=0$, where $s_j$ are holomorphic sections such that 
$|s_j|<1/e$. Consider the following density
$$
f = \frac{1}{\prod_{j=1}^N |s_j|^{2}}. 
$$
Then for suitable positive constants $C_1, C_2$ the following function 
$$
\f:= -2\sum_{j=1}^N \log (-\log |s_j|+C_1) -C_2
$$
is a subsolution of $\MA(\f)=e^{\f}f\omega^n$. In fact, it suffices to find a function $u\in \Ec(X,\omega/2)$ such that $e^{u}f$ is integrable (see Example
\ref{ex: subsol semi}).
\end{ex}

\subsection{The case of semipositive and big classes}
In this section we try to extend our result in Theorem C to the case of semipositive and big classes. Let $\theta$ be a smooth closed semipostive $(1,1)$-form
on $X$ such that $\int_X \theta^n>0$. Assume that  $E=\sum_{j=1}^M a_j E_j$ 
is an effective simple normal crossing divisor on $X$ such that $\{\theta\}-c_1(E)$
is ample. Let $0\leq f$ is a non-negative measurable function on $X$. Consider the following degenerate complex Monge-Amp\`ere equation
\begin{equation}\label{eq: MA semi}
(\theta +dd^c \f)^n = e^{\f}f\omega^n.
\end{equation}
As in Theorem C we obtain here a similar regularity for solutions in $\EcX$:

\begin{thm}\label{thm: thmC semi}
Assume that $0<f\in \Cc^{\infty}(X\setminus D)$ satisfies Condition $\Hc_f$. Let 
$\theta$ and $E$ be  as above. If there is a solution in $\EcX$ of equation (\ref{eq: MA semi}) then this solution is also smooth on $X\setminus (D\cup E)$. 
\end{thm}
Note that in Theorem \ref{thm: thmC semi} we do not assume that $f$ is integrable on $X$. We also stress that there is at most one solution in $\Ec(X,\theta)$ (see \cite{BG13}). 
\begin{proof}
We adapt the proof of Theorem 3 in \cite{DL1} where we followed essentially the ideas in \cite{BEGZ10}. Assume that $\f\in \Ec(X,\theta)$ 
is a solution to equation (\ref{eq: MA semi}). By assumption on $f$ we can find 
a uniform constant $C>0$ such that 
$$
f=e^{\p^+-\p^-}, \ \ dd^c \p^{\pm} \geq -C \omega^n ,\ \ \sup_X \p^+\leq C, \ \sup_X \f \leq C, \ \ \p^-\in L^{\infty}_{\rm loc}(X\setminus D).
$$
We regularize $\f$ and $\p^{\pm}$ by using the smoothing kernel $\rho_{\vep}$ in Demailly's work \cite{Dem92}. Then for 
$\vep>0$ small enough we have 
$$ 
dd^c \rho_\varepsilon(\p^-) \geq -C_1 \omega,\;\; 
dd^c \rho_\vep(\f+\psi^+)\geq -C_1 \omega, \;\; 
\sup_X \rho_\vep(\f+\psi^+) \leq C_1,
$$
where $C_1$ depends on $C$ and  the Lelong numbers of the currents $C\omega+dd^c \psi^{\pm}$. For each $\vep>0$ by the famous result of Yau \cite{Y78} there exits a unique smooth $\phi_{\vep}\in \psh(X,\theta+\vep \omega)$ normalized by 
$\sup_X \phi_{\vep}=0$ such that 
$$
(\theta + \vep \omega +dd^c \phi_{\vep})^n = e^{c_{\vep}+\f_{\vep}+\p^+_{\vep}-\p^-_{\vep}}\omega^n = g_{\vep} \omega^n,
$$
where $c_{\vep}$ is a normalized constant. As in the proof of Theorem 3 in \cite{DL1} we can prove that $c_{\vep}$ converges to $0$ as $\vep \downarrow 0$.
We then can show that $\phi_{\vep}$ converges in $L^1$ to $\f-\sup_X \f$. Now, 
we can apply Theorem 5.1 and Theorem 5.2 in \cite{DL1} to get  uniform 
bound on $\phi_{\vep}$ and $\Delta_{\omega} \phi_{\vep}$ on every compact subset 
of $X\setminus (D\cup E)$. From this we can  get the smoothness of $\f$ on $X\setminus (D\cup E)$ as in \cite{DL1}.
\end{proof}

It follows from Theorem \ref{thm: subsol} (which is also valid in the case of semipositive and big classes) that to solve the equation it suffices to find a subsolution in $\Ec(X,\theta)$.
We show in the following example that in some cases 
it is easy to find a subsolution in $\Ec(X,\theta)$.

\begin{ex}\label{ex: subsol semi}
We consider the density given in Example \ref{ex: subsol}. Assume that $\theta$
satisfies $\{\theta\}-c_1(E)>0$, where $E=\sum_{j=1}^M{a_j E_j}$ is 
an effective simple normal crossing divisor on $X$. Assume that $E_j$ is defined by the zero locus 
of a holomorphic section $\sigma_j$ such that $|\sigma_j|<1/e$. Then for some constants $p\in (0,1)$ and $a>0$, $A\in \R$ the following function
$$
u:= -\left(-a\sum_{j=1}^N\log |s_j| - \frac{1}{2}\sum_{j=1}^M a_j \log |\sigma_j|\right)^p-A
$$
belongs to $\Ec(X,\theta/2)$ and verifies 
$\int_X e^{u}f\omega^n=2^{-n}\int_X \theta^n$. It follows from \cite{BBGZ13} that there exists  $v\in \Ec(X,\theta/2)$ such that $v\leq 0$ and 
$$
(\theta/2 + dd^c v)^n = e^{u}f\omega^n.
$$
It is easy to see that $\f:=u+v\in \Ec(X,\theta)$ is a subsolution of (\ref{eq: MA semi}).
\end{ex}

\subsection{Critical Integrability}
Recently, Berndtsson \cite{Ber13} solved the openness conjecture of Demailly
and Koll\'ar \cite{DemK01} which says that given $\phi\in \psh(X,\omega)$ and 
$$
\alpha(\phi) = \sup\{t>0 \ \setdef \ e^{-t\phi} \in L^1(X)\} <+\infty,
$$
then one has $e^{-\alpha \phi} \notin L^1(X)$ (a stronger version of the openness conjecture has been quite recently obtained by Guan and Zhou \cite{GZ13}).

In the following result, we use the generalized capacity to show that $e^{-\alpha\phi}$ is however not far to be integrable in the following sense:
\begin{thm}\label{thm: critical integrability}
Let $\phi\in \psh(X,\omega)$ and $\alpha=\alpha(\phi)\in (0,+\infty)$ be the canonical threshold of $\phi$. Then we can find $\f\in \psh(X,\omega)$  having zero Lelong number at all points of $X$
such that 
$$
\int_X e^{\f-\alpha \phi} \omega^n <+\infty. 
$$
\end{thm}
One can moreover chose $\f=\chi\circ \phi\in \Ec(X,\omega)$ for some $\chi$ increasing convex function. We thank S. Boucksom and H. Guenancia for indicating this.
\begin{proof}
Let $\alpha_j$ be an increasing sequence of positive numbers which converges to $\alpha$. By assumption we have $e^{-\alpha_j \phi}$ is integrable on $X$.  We can assume that $\phi\leq 0$. We solve the complex Monge-Amp\`ere equation
$$
(\omega +dd^c \f_j)^n = e^{\f_j-\alpha_j\phi}\omega^n.
$$
For each $j$, since $e^{-\alpha_j \phi}$ belongs to $L^{p_j}$ for some $p_j>1$, it follows from the classical result of Ko{\l}odziej \cite{Kol98} that 
$\f_j$ is bounded. Moreover, the comparison principle reveals that 
$\f_j$ is non-increasing. Now, we need to bound $\f_j$ uniformly from below by some singular quasi-psh function. 

Let $1/2>\vep>0$ be a very small positive number. By assumption 
we know that 
$$
e^{(\vep-\alpha)\phi} \in L^p(X), \ \ p=p_{\vep}:=\frac{\alpha-\vep/2}{\alpha-\vep}>1.
$$
Set $\p:=\vep \phi\in \psh(X,\omega/2)$ and consider the function
$$
H(t):= \left[\Capis(\f_j<\p-t)\right]^{1/n}, \ \ t>0.
$$
It follows from \cite[Lemma 2.7]{DL1} that 
$\Capo\leq 2^n \Capa_{\psi}$. By a strong volume-capacity domination in \cite[Remark 5.10]{GZ05} we also have 
$$
\vol_{\omega} \leq \exp{\left(\frac{-C_1}{\Capo^{1/n}}\right)},
$$
where $C_1$ depends only on $(X,\omega)$. 
Fix $t>0, s\in [0,1]$. Using \cite[Proposition 2.8]{DL1} and H\"older inequality  we get
\begin{eqnarray*}
s^n \Capis(\{\varphi_j<\psi-t-s\})&\leq & \int_{\{\varphi_j<\p-t\}}\MA(\varphi_j) \\
&\leq & \int_{\{\varphi_j<\psi-t\}}e^{-\f_j} e^{\p} \MA(\f_j) \\
&\leq & \int_{\{\f_j<\p-t\}} e^{(\vep-\alpha)\phi} \omega^n\\
&\leq & \left[\int_X e^{(\vep/2-\alpha)\phi} \omega^n\right]^{1/p}
\left[\int_{\{\f_j<\psi-t\}} \omega^n \right]^{1/q}\\
&\leq & C_2 \left[\Capis(\{\f_j<\psi-t\})\right]^2,
\end{eqnarray*}
where $p= p_{\vep}:=\frac{\alpha-\vep/2}{\alpha-\vep}>1$ and $q>1$ is 
the exponent conjugate of $p$. The constant $C_2>0$ depends on $\vep$
and also on $\int_X e^{(\vep/2-\alpha)\phi} \omega^n$. We then get
$$ 
s H(t+s)\leq C_2^{1/n} H(t)^2,  \ \forall t>0, \forall s\in [0,1]. 
$$
Then by \cite[Lemma 2.4]{EGZ09} we get 
$$
\f_j \geq \vep \phi-C_{\vep},
$$
where $C_{\vep}$ only depends on $\vep$ and $\int_X e^{(\vep/2-\alpha)\phi} \omega^n$. Then we see that $\f_j$ decreases to 
$\f\in \psh(X,\omega)$ and $\f$ satisfies 
$$
\f\geq \vep \phi -C_{\vep}.
$$
Since $\vep$ is arbitrarily small we conclude that $\f$ has zero Lelong number everywhere on $X$. Finally, it follows from Fatou's lemma that $e^{\f-\alpha \phi}$ is integrable on $X$. \\
\indent We now show that $\f$ can be chosen to be in $\Ec(X,\omega)$, more precisely $\f=\chi\circ \phi$,
$$
\int_X e^{\chi\circ \phi -\alpha \phi} \omega^n <+\infty,
$$
for some $\chi: \R^-\rightarrow \R^-$ increasing convex function such that 
$\chi(-\infty)=-\infty$ and $\chi'(-\infty)=0$. Note that $\chi\circ \phi\in \Ec(X, \omega)$ thanks to \cite{CGZ07}. We are grateful to H. Guenancia for the following constructive proof.\\ 
We can always assume that $\phi \leq -1$. For each $k\in \N$ let 
\begin{equation}\label{eq: construction 1}
a_k:= \log \int_X e^{-(\alpha-2^{-k-1})\phi}\omega^n <+\infty .
\end{equation}
Define the sequence $(c_k)$ inductively by 
\begin{equation}\label{eq: construction 2}
c_1=a_1, \ c_{k+1}: = \max (c_k +4k , a_{k+1}), \ \forall k \geq 1 .
\end{equation}
Define another sequence $(t_k)$ by
\begin{equation}\label{eq: construction 3}
t_1:=1, \ t_{k+1}:= 2^{k+1}(c_{k+1}-c_k), \ \forall k\geq 1 .
\end{equation}
Define $\chi: (-\infty,-1]\rightarrow \R^-$ by 
$$
\chi(-t) := -2^{-k}t -c_k \ \ {\rm if}\ \ t\in [t_k, t_{k+1}], \ \forall 
k\geq 1 .
$$
It follows from (\ref{eq: construction 1}) that 
$$
e^{(\alpha -2^{-k-1})t}\ \vol (\phi<-t) \leq \int_X e^{-(\alpha -2^{-k-1})\phi}\omega^n \leq e^{c_k}.
$$
Thus using  (\ref{eq: construction 2}), (\ref{eq: construction 3}) and the above inequality  we get  
\begin{eqnarray*}
\int_X e^{\chi(\phi) -\alpha \phi} \omega^n &\leq 
& e^{\chi(-1)+\alpha} + 
\alpha \int_1^{+\infty} e^{\alpha t +\chi(-t)}\vol (\phi<-t) dt\\
&\leq & C + \alpha\sum_{k=1}^{+\infty} \int_{t_k}^{t_{k+1}} e^{\alpha t + \chi(-t)} \
\vol (\phi<-t) dt\\
&\leq & C+ \alpha\sum_{k=1}^{+\infty} \int_{t_k}^{t_{k+1}} e^{c_k+2^{-k-1}t-2^{-k}t-c_k} dt\\
&\leq & C+ \alpha\sum_{k=1}^{+\infty}\int_{t_k}^{t_{k+1}} e^{-2^{-k-1}t} dt\\
&\leq & C+ \alpha\sum_{k=1}^{+\infty}2^{k+1} e^{-2^{-k-1}t_k} \\
&\leq & C+ \alpha\sum_{k=1}^{+\infty} 2^{k+1} e^{-2^{-1}(c_k-c_{k-1})} \\
&\leq & C+ \alpha\sum_{k=1}^{+\infty} 2^{k+1} e^{-2(k-1)} \\
&\leq & C+ 4\alpha .
\end{eqnarray*}

\end{proof}

The above result is quite optimal as the following example shows:
\begin{ex}\label{example: critical optimal}
Let $(X,\omega)$ be a compact K\"ahler manifold and $D$ be a smooth complex hypersurface on $X$ defined by a holomorphic section $s$ such that $\vert s\vert \leq 1/e$.
Consider 
\begin{equation}\label{eq: critical int}
\phi = 2\log \vert s\vert -(-\log \vert s\vert)^{p}, \ \ p\in (0,1).
\end{equation}
By rescaling $\omega$ we can assume that $\phi\in \psh(X,\omega)$.
Then for any $q>0$ 
$$
\int_X \frac{e^{-\phi}}{(-\phi)^q}\omega^n =+\infty.
$$
\end{ex}

The example above has been given in \cite{ACK} in the case of one complex variable which is locally similar to our setting. Assume now that $\phi$
is given by (\ref{eq: critical int}). It follows from Theorem
\ref{thm: critical integrability} that we can find $\f\in \psh(X,\omega)$ having zero Lelong number everywhere such that 
$$
\int_X e^{\f-\phi} \omega^n <+\infty.
$$
In this concrete example one such function $\f$ can be given explicitly by
$$
\f = -(\log |s|)^p -(1+\vep)\log (\log |s|), \ \vep>0.
$$

\begin{proof}[Proof of Theorem D]
It follows from the above proof of Theorem \ref{thm: critical integrability} that there exists $u\in \Ec(X,\omega/2)$ such that $e^{u-\alpha\phi}$ is integrable. We then can argue as in Example \ref{ex: subsol semi} to find a subsolution which also yields a solution thanks to Theorem \ref{thm: subsol}. The uniqueness follows from the comparison principle (see \cite{BG13}).
\end{proof}

\end{document}